\documentclass[a4paper,fleqn]{article}
\usepackage{times,cite,amssymb, amsmath, amsthm}
\usepackage[T1]{fontenc}
\usepackage[utf8]{inputenc}
\usepackage{fixme}
\usepackage[margin=0.5in]{geometry}
\fxsetup{%
	draft,%
	layout = inline,%
	theme=color,%
	silent = true,%
	}
\theoremstyle{plain}

\theoremstyle{theorem}

\newtheorem{thm}{Theorem}

\usepackage{graphicx}
\begin{document}

%% \def\leftmark{Session title}
%%
%%    The information for the title page will be placed between
%%    \begin{document} and \maketitle. The order of most entries
%%    is determined by the class file and can not be changed by
%%    rearranging them. The maketitle command follows after the
%%    abstract.
%%
%%    Most of the following commands will be completed by the publisher.
%%
%%    \renewcommand{\copyrightyear}{2016}
%%    \DOIsuffix{pamm.20161zzzz}
%%    \Volume{16} 
%%    \Year{2016} 
%%    \pagespan{1}{}
%%
%%    The short title is optional:

%\TitleLanguage[EN]
\title%[LSFEM for sinusiudal flow]
{Least Squares Finite Element Method for Hepatic Sinusoidal Blood Flow}

%% Please do not enter footnotes or \inst{}-notes into the optional
%% argument of the author command. 

%% Please delete not needed author entries.
%% Information for the first author.
%\author{\firstname{Fleurianne}  \lastname{Bertrand}\inst{1,}%
% \footnote{Corresponding author: fb@math.hu-berlin.de \ElectronicMail{fb@math.hu-berlin.de},
 %       }}
\author{F. Bertrand, L. Lambers, T. Ricken}
%\address[\inst{1}]{\CountryCode[DE]Institut für Mathematik, Humboldt-Universität zu Berlin, Unter den Linden 6, 10099 Berlin, Germany}
%%%
%%%    Information for the second author
%\author{\firstname{Lena} \lastname{Lambers}\inst{2,}%
%  %\footnote{Second author footnote.}
%  }
%
%\address[\inst{2}]{\CountryCode[DE]Institute of Mechanics, Structural Analysis and Dynamics, Faculty of Aerospace Engineering and Geodesy, University of Stuttgart, 70569 Stuttgart, Germany}
%%%
%%%    Information for the third author
%\author{\firstname{Tim} \lastname{Ricken}\inst{2,}%
%  %\footnote{Third author footnote.}
%  }
%%%
%%%    \dedicatory{This is a dedicatory.}
%%%
%%%    Abstract is required.
%\AbstractLanguage[EN]

%% maketitle must follow the abstract.
\maketitle                   % Produces the title.
\begin{abstract}
The simulation of complex biological systems such as the description of blood flow in organs requires a lot of computational power as well as a detailed description of the organ physiology. We present a %model
novel Least-Squares discretization method
for the simulation of sinusoidal blood flow in liver lobules using a porous medium approach for the liver tissue. The scaling of the different Least-Squares terms leads to a robust algorithm and the inherent error estimator provides an efficient refinement strategy.
\end{abstract}
\newcommand{\bv}{\mathbf v}
\newcommand{\fluidvelocity}{\mathbf{v}_{f}}
\newcommand{\hydrostaticpressure}{{p}_{f}}
\newcommand{\fluidpressure}{{p}_{F}}
\newcommand{\filtervelocity}{\mathbf{v}_{p}}
\renewcommand{\filtervelocity}{(n^F\mathbf{x}_{F}^{\prime})}
\newcommand{\shearstress}{\boldsymbol \tau}

% double porosity
% sinusoid in leberläbchen

\section{Introduction}
Modeling the hepatic blood perfusion is challenging \cite{Christ2010} but a crucial part for the understanding of the complex microcirculation processes and the development of treatment of possible liver diseases. Performing large scale simulations of the microcirculation arises in applications of paramount importance involving the supply of hepatocytes (liver cells) with oxygen, nutrients or pharmaceuticals. 
However, a fully spatially resolved model of the complex microcirculation of the hepatic lobule, the functional unit of the liver, on a very fine mesh would exceed the limits of current computational power and time. 
An inherent error estimator as in the Least-Squares method then constitutes a major advantage. 
The continuum biomechanical model is based on \cite{Ricken2010} and \cite{Ricken2015}, where a  multicomponent, multiscale and multiphase homogenization model in the framework of the Theory of Porous Media (TPM) for blood micro-circulation in hepatic lobules is presented. Taking into account the porous structure and the hexagonal shape with inflow at the outer edges (portal triad) and an outflow at the center (central vein), cf. Figure \ref{fig:5}c, a Least-Squares approach with suitable boundary conditions is validated using a single liver lobule. Least-Squares finite element methods are an attractive class of methods for the numerical solution of partial differential equations, as they produce symmetric and positive definite discrete systems, possess an inherent error estimator and are relatively easy to apply on nonlinear systems.
We refer to \cite{LS} for a comprehensive overview.

\section{Modeling Hepatic Lobular Sinusoidal Blood Flow}
\label{modeling}
The multi-scale structure of the liver consists of cells, lobules, segments, and lobes. The liver cells (hepatocytes) are arranged along so-called sinusoids. Sinusoids are capillary-like small blood vessels in the liver and form the liver lobules, the functional unit of the liver. The liver structure is a highly complex structure that can be described as a porous medium similar to a capillary bed as in \cite{Koch2014CouplingAV}.  
There a porous medium is typically a multiphase system where a mixture $\varphi = \cup_{\alpha \in \{ S, F\} }  $ with a solid phase $\varphi^S$ and the pore fluid $\varphi^F$ {is assumed}.\begin{figure}[h]
\begin{minipage}{65mm}
\vspace{-8cm}
\hspace{0.38cm} \includegraphics[width=45mm]{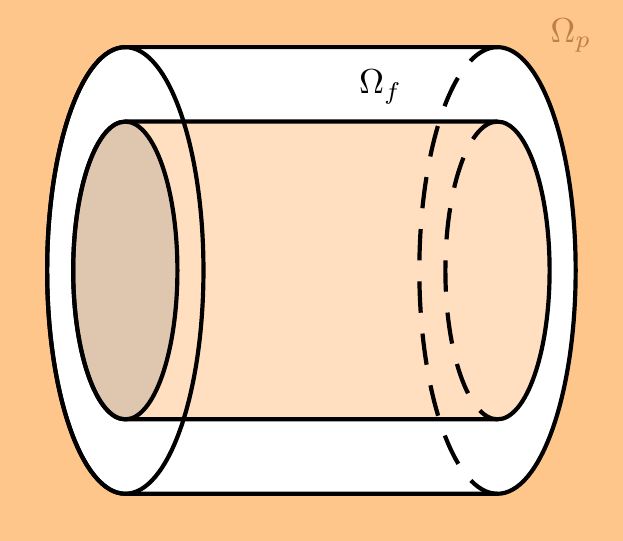}~a)
\\\noindent\includegraphics[width=55mm]{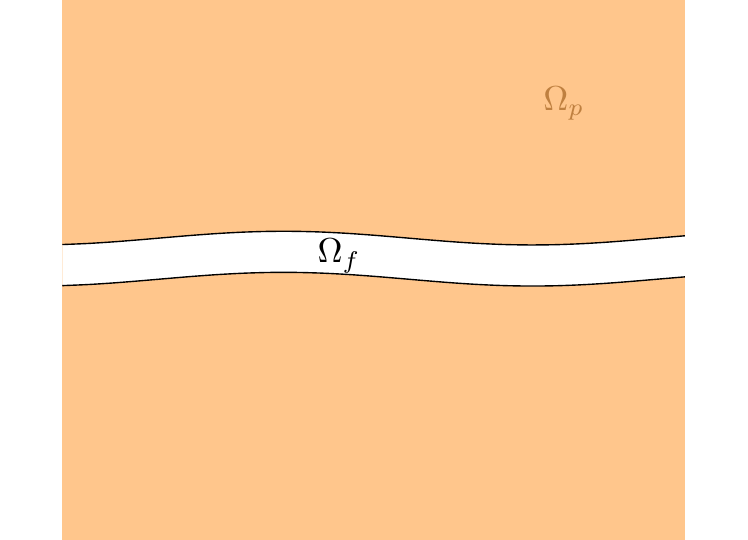}~\hspace*{-0.4cm}b)
\end{minipage}
\hfil
\includegraphics[width=7cm]{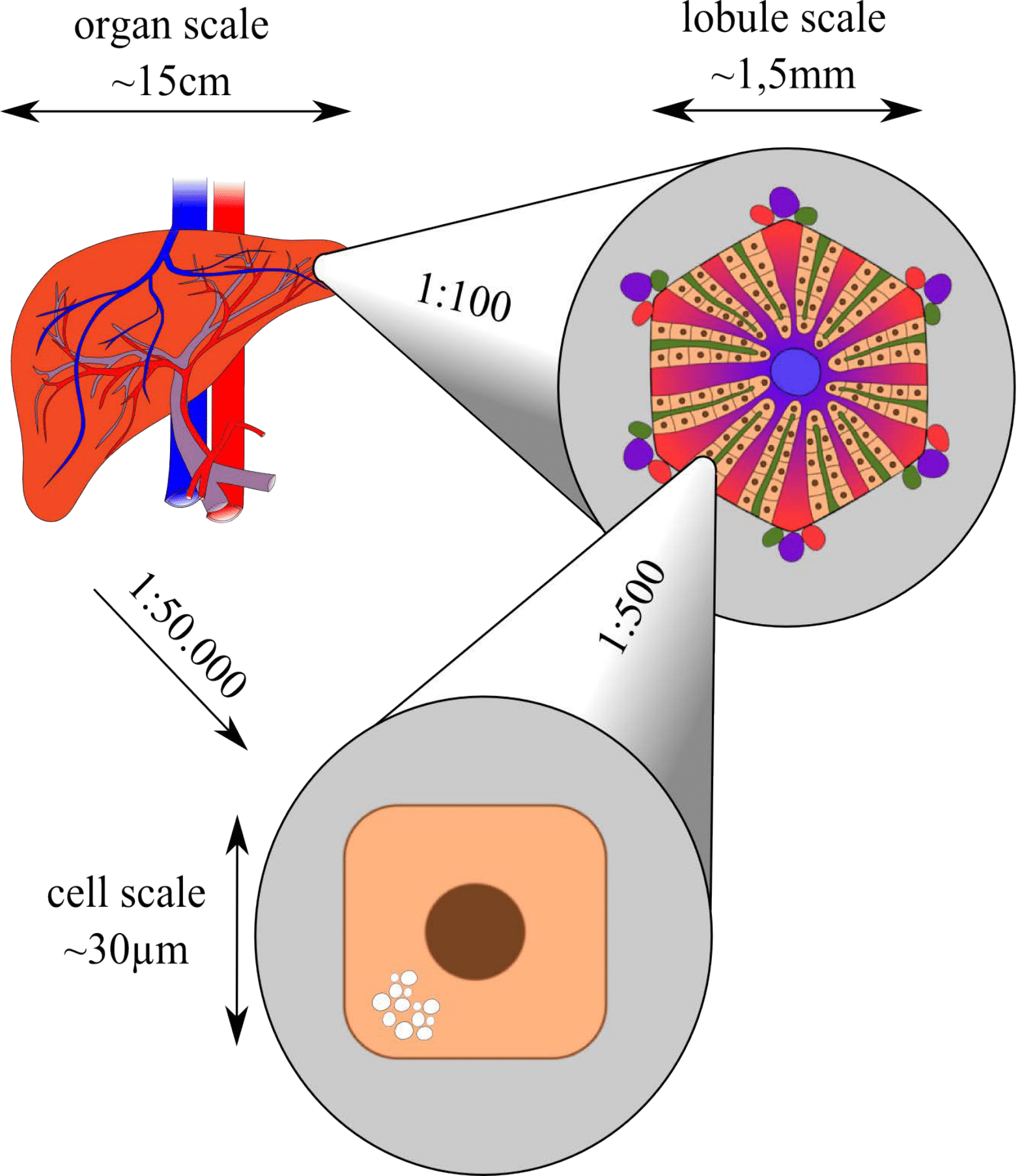}~c)
\caption{\textbf{a} and \textbf{b} : different views of free-flow domain $\Omega_f$ and porous domain $\Omega_p$ and \textbf{c} Different size scales of the human liver namely the organ scale including the macroscopic vascular, the porous lobule scale with a hexagonal shape and blood flow through the sinusoids and the cell scale.}
\label{fig:5}
\end{figure}

Each constituent $\alpha\in \{ S, F\} $ of the mixture is described by an individual motion function $\chi_\alpha$ and velocity $\bf x_\alpha^\prime$ and has
a partial density $\rho^\alpha$. 
 The local composition of the mixture is described by partial volumes $V^\alpha$ and volume fractions $n^\alpha$.
In our transvascular model we have constant solidity $n^S$ and the porosity $n^F$, as well as 
$
\sum_{\alpha} n^{\alpha}=1
$ and \(\rho=\sum_{\alpha} \rho^{\alpha}\).
The mass balance of a constituent $\alpha$ reads
\(\frac{\partial \rho^{\alpha}}{\partial t}+\nabla \cdot\left(\rho^{\alpha} \mathbf{x}_{\alpha}^{\prime}\right)=\hat{\rho}^{\alpha}=0\) where $\hat{\rho}^{\alpha}$ is a production term that accounts for interaction with the other constituents. Since here the two  constituents are immiscible, $\hat{\rho}^{\alpha}$ vanishes. We further assume a rigid body motion as well as a quasi static description with $\textbf x^{\prime \prime}_{\alpha} = \textbf 0$. 
The balance of momentum for the constituent $\alpha$ reads
\begin{align}
\rho^{\alpha}\left(\frac{\partial \mathbf{x}_{\alpha}^\prime}{\partial t}+\mathbf{x}_{\alpha}^{\prime} \cdot \nabla \mathbf{x}_{\alpha}^{\prime}\right)=\nabla \cdot \mathbf{T}^{\alpha}+\rho^{\alpha} \mathbf{f}^{\alpha}+
\hat{\textbf{p}}^{\alpha}+
\hat{\rho}^{\alpha}\mathbf{x}_{\alpha}^{\prime}\end{align}
where $\hat{\textbf{p}}^{\alpha}$ accounts for the momentum production by interaction with other constituents. The stress tensor for a general fluid phase is given by 
\(\mathbf{T}^{F}=\mathbf{T}_{\mu}^{F}-n^F p_F {\bf I}=2 \mu_{F} \mathbf{D}_{F}+\lambda\left(\mathbf{D}_{F} \cdot {\bf I}\right) {\bf I}-n^F p_F {\bf I}\)
with the second Lamé constant $\lambda$.
The fluid in the porous medium follows the constitutive law and thus is given by
$\hat{\textbf{p}}^F =  p_F \nabla n^F-{n^F}^{2} \mu^{F} \mathbf{K}^{-1}\left(\mathbf{x}_{F}^{\prime}-\mathbf{x}_{S}^{\prime}\right),$ with the fluid pressure $p_F$, $n^F$ denotes the porosity, $\mu^F$ the dynamic viscosity of the interstitial fluid and $K$ the permeability of the porous medium. The mass balance of the interstitial fluid reduces to $\nabla \dot  {(n^F \mathbf{x}^{\prime}_{F})}=0$. Thus, the momentum balance for the interstitial fluid leads to
$
{(n^F \mathbf{x}^{\prime}_{F})}=-\frac{\mathbf{K}}{\mu^{F}} \nabla p_F \ .
%\qquad \qquad \qquad
%-\nabla \cdot\left(\frac{\mathbf{K}}{\mu^{F}} \nabla p\right)=0
$
The domain representing the porous capillary bed is denoted by $\Omega_p$, while the domain $\Omega_f$ represents the blood vessel. Moreover, we denote $\Omega= \Omega_f \cup \Omega_p$ and $\Gamma = \Omega_f \cap \Omega_p$. 
The blood in the sinusoid is a mixture of several components and its  behavior is generally non-Newtonian but can be assumed to behave as a Newtonian fluid. Moreover, as the linearisation of Least-Squares methods is mostly straight-forward, we assume that the shear stress tensor $\shearstress$ is given by $\shearstress = 2\mu \mathbf{D}(\mathbf{v}_f)$ such that we obtain
\begin{equation}
 -2 \mu \nabla \cdot \mathbf{D}(\mathbf{v}_f)+\nabla p_f =0 \qquad\qquad \nabla \cdot \mathbf{v}_f =0 \qquad \qquad 
\mathbf{D}(\mathbf{v}_f)=\frac{1}{2}\left(\nabla \mathbf{v}_f+\nabla^{T} \mathbf{v}_f\right)
\text{ in } \Omega_f \ 
\end{equation}
with the hydrostatic pressure $p_f$. It remains to find interface conditions to couple the domain $\Omega_f$ and $\Omega_p$ representing a selective permeable membrane and leading to the well-posedness of the problem. Following \cite{Koch2014CouplingAV} we describe the vessel wall without spatially resolving it : we neglect the slip velocity $\fluidvelocity\cdot \boldsymbol{\tau} = 0$, consider the continuity of the normal velocity such that $\fluidvelocity\cdot \mathbf{n}=\filtervelocity \cdot \mathbf{n} $ and the Beavers-Joseph-Saffman condition $$ -2 \mu \mathbf{D}\left(\fluidvelocity\right) \mathbf{n} \cdot \mathbf{n}+\hydrostaticpressure=\frac{\mu_{F} \varepsilon}{K_{{M}}} \fluidvelocity \cdot \mathbf{n}+\fluidpressure$$
with the thickness of the vessel wall / sinusoid  %dM!%
$\varepsilon$ and the intrinsic permeability of the capillaries $K_M$. We obtain the following coupled system.
\begin{equation}
\label{system}
\begin{aligned}
-2 \mu \nabla \cdot \mathbf{D}\left(\fluidvelocity\right)+\nabla \hydrostaticpressure=0 
&\quad \text{ in } \Omega_{f} \\
 -\nabla \cdot \fluidvelocity=0 
 &\quad \text{ in }  \Omega_{f} 
 \qquad & \fluidvelocity\cdot \mathbf{n}=\filtervelocity \cdot \mathbf{n} 
 &\quad \text{ on }  \Gamma 
 \\ 
 \frac{\mu_F}{K} \filtervelocity+\nabla \fluidpressure=0 
 &\quad \text{ in }  \Omega_{p} 
 \qquad &  -2 \mu \mathbf{D}\left(\fluidvelocity\right) \mathbf{n} \cdot \mathbf{n}+\hydrostaticpressure=\frac{\mu_{F} \varepsilon}{K_{\mathcal{M}}} \fluidvelocity \cdot \mathbf{n}+\fluidpressure &\quad\text{ on }  \Gamma 
 \\ 
 -\nabla \cdot \filtervelocity=0 &\quad \text{ in } \Omega_{p} 
 & \qquad 
\fluidvelocity\cdot \boldsymbol{\tau}=0 & \quad\text{ on }  \Gamma
\end{aligned}
\end{equation}
This is in fact a Darcy- Stokes system as reviewed in \cite{Discacciati}. Although for the Least-Squares method this is similar to \cite{steffendarcyflow2015}, the additional coupling term $\frac{\mu_{i} \varepsilon}{K_{\mathcal{M}}} \fluidvelocity \cdot \mathbf{n}$ has to be taken into account. Moreover, the different scaling require special care.

%%%%%%%%%%%%%%%%%%%%%%%%%%%%%%

\section{Least Squares Finite Element Method}
In order to apply a Least-Squares Method to a first-order system corresponding to the system \eqref{system}, we introduce the stress tensor $\boldsymbol \sigma_f = 2\mu\mathbf{D}\left(\fluidvelocity\right)- \hydrostaticpressure \mathbf I$. Therefore, the first equation in \eqref{system} becomes $\nabla \cdot \boldsymbol \sigma_f = 0$. Moreover, with the deviator operator \(\operatorname{dev} \boldsymbol{\sigma}=\boldsymbol{\sigma}-(1/2)(\operatorname{tr} \boldsymbol{\sigma}) \mathbf{I}\) the definition of the stress tensor becomes $\operatorname{dev} \boldsymbol{\sigma}_f = 2 \mu  \mathbf{D}\left(\fluidvelocity\right)$ due to the incompressibility of $\fluidvelocity$. The incompressibility of $\fluidvelocity$ also implies that the hydrostatic pressure is directly related to the trace of the stress tensor : $\fluidpressure =-(1/2)(\operatorname{tr} \boldsymbol{\sigma}_f)$. 
\renewcommand{\hydrostaticpressure}{-(1/2)(\operatorname{tr} \boldsymbol{\sigma}_f)}In fact, both equations are sufficient to ensure the well-posedness of the method, see e.g. \cite{bertrandinterface}. Finally, we replace the filter velocity $\filtervelocity$ 
by\renewcommand{\filtervelocity}{\bv_p}
\newcommand{\stress}{\boldsymbol{\sigma}_f}
$\bv_p$. The first-order system therefore reads
\begin{equation}
\label{system2}
\begin{aligned}
 -\nabla \cdot \stress=0 &\quad \text{ in } \Omega_{f}   \qquad & 
\operatorname{dev} \stress = -2 \mu \nabla \cdot \mathbf{D}\left(\fluidvelocity\right)
&\quad \text{ in } \Omega_{f} \\
 -\nabla \cdot \fluidvelocity=0 
 &\quad \text{ in }  \Omega_{f} 
 \qquad & \fluidvelocity\cdot \mathbf{n}=\filtervelocity \cdot \mathbf{n} 
 &\quad \text{ on }  \Gamma 
 \\ 
 \frac{\mu_F}{K} \filtervelocity+\nabla \fluidpressure=0 
 &\quad \text{ in }  \Omega_{p} 
 \qquad &  -2 \mu \mathbf{D}\left(\fluidvelocity\right) \mathbf{n} \cdot \mathbf{n}-\frac 1 2 (\operatorname{tr} \boldsymbol{\sigma}_f)=\frac{\mu_{F} d_{M}}{K_{\mathcal{M}}} \fluidvelocity \cdot \mathbf{n}+\fluidpressure &\quad\text{ on }  \Gamma 
 \\ 
 -\nabla \cdot \filtervelocity=0 &\quad \text{ in } \Omega_{p} 
 & \qquad 
\fluidvelocity\cdot \boldsymbol{\tau}=0 & \quad\text{ on }  \Gamma
\end{aligned}
\end{equation}

We use the standard notation and definition for the Sobolev spaces $L^2(\Omega)$, $H^s(\Omega)$ for $s\geq 0$ and $$H(\operatorname{div} ; \Omega)=\left\{\boldsymbol{\tau} \in L^{2}(\Omega)^{d}: \nabla \cdot \boldsymbol{\tau} \in L^{2}(\Omega)\right\}.$$ We assume that the boundary of the domain $\Omega$ is split into a Dirichlet part $\Gamma_D$ and a Neumann 
part $\Gamma_N$ and write $\Omega_{F,D}$ for the intersection of $\Gamma_D$ with $\Omega_p$ where an effective pressure $p_f^0$ is imposed and  $\Omega_{f,D}$ for the intersection of $\Gamma_D$ with $\Omega_S$ where an effective pressure $p_F^0$ is imposed. In the numerical example, $p_f^0$ and $p_F^0$ are piecewise constants such that these conditions can be built directly in the finite element space. Therefore, the Least-Squares methods seeks $(\boldsymbol{\sigma}_f, {\bv}_f ,{\bv}_F,p_F )\in \mathbb H := H_{\Gamma_N}(\operatorname{div} ; \Omega_f)^2
\times H^1_{\Gamma_D}(\Omega_f)^2 
\times H_{\Gamma_N}(\operatorname{div} ; \Omega_p)
\times H^1_{\Gamma_D}(\Omega_p)$ such that the Least-Square Functional 
\begin{equation}\label{eq:lsF}\begin{aligned}
\mathcal F (\boldsymbol{\sigma}_f, {\bv}_f ,{\bv}_F,p_F )
=&
\| \nabla \cdot \stress\|_{\Omega_{f}} ^2 +
\| \operatorname{dev} \stress +2 \mu \nabla \cdot \mathbf{D}\left(\fluidvelocity\right)\|_{\Omega_{f}}^2 +
 \| \nabla \cdot \fluidvelocity\|_{\Omega_{f}} ^2
+ \|\nabla \cdot \filtervelocity\|_
{\Omega_{p} }^2
\| \frac{\mu_F}{K} \filtervelocity+\nabla \fluidpressure \|_
{\Omega_{p} }^2\\&
+\|2 \mu \mathbf{D}\left(\fluidvelocity\right) \mathbf{n} \cdot \mathbf{n}+\frac 1 2 (\operatorname{tr} \boldsymbol{\sigma}_f)+\frac{\mu_{F} d_{M}}{K_{\mathcal{M}}} \fluidvelocity \cdot \mathbf{n}+\fluidpressure \|_{\Gamma}^2
+\|  \fluidvelocity\cdot \mathbf{n}-\filtervelocity \cdot \mathbf{n} 
\|_{-\frac 1 2 , \Gamma}^2
\end{aligned}
\end{equation}
is minimized in $\mathbb H$.
Note that the condition $\fluidvelocity\cdot \boldsymbol{\tau}=0$ is imposed directly in the space $H^1_{\Gamma_D}(\Omega_f)^2$.
The underlying result for the success of the numerical method is the following continuity and ellipticity of the Least-Squares functional.

\begin{thm}
The Least-Squares functional defined in \eqref{eq:lsF} is continuous and elliptic in $\mathbb H$, equipped with the norm \begin{equation}
|||(\boldsymbol{\sigma}_f, {\bv}_f ,{\bv}_F,p_F )|||^2
:= \|\boldsymbol{\sigma}_f \|_{\text{div},\Omega_f}^2 + \|  {\bv}_f\|_{1,\Omega_f} ^2
+ \| {\bv}_p \|_{\text{div},\Omega_p}^2 + \|p_F \|_{1,\Omega_p}^2  \ .
\end{equation}
\end{thm}
\begin{proof}
The proof is similar to the one in \cite{bertrandinterface},  \cite{steffendarcyflow2015} and \cite{steffengerharddarcyflow2011}. Special care is needed for the term $\|2 \mu \mathbf{D}\left(\fluidvelocity\right) \mathbf{n} \cdot \mathbf{n}+\frac 1 2 (\operatorname{tr} \boldsymbol{\sigma}_f)+\frac{\mu_{F} d_{M}}{K_{\mathcal{M}}} \fluidvelocity \cdot \mathbf{n}+\fluidpressure \|_{\Gamma}^2$ in the ellipticity proof. In fact, it follows from \cite{steffengerharddarcyflow2011} Lemma 3.2. that
\begin{equation}\begin{aligned}
\frac{1}{\eta}&
\left(\|2 \mu \mathbf{D}\left(\fluidvelocity\right) \mathbf{n} \cdot \mathbf{n}+\frac 1 2 (\operatorname{tr} \boldsymbol{\sigma}_f)+\frac{\mu_{F} d_{M}}{K_{\mathcal{M}}} \fluidvelocity \cdot \mathbf{n}+\fluidpressure \|_{\Gamma}^2
+\|  \fluidvelocity\cdot \mathbf{n}-\filtervelocity \cdot \mathbf{n} 
\|_{-\frac 1 2 , \Gamma}^2
\right)
+\eta \|  \fluidvelocity\cdot \mathbf{n} \|_{0,\Gamma}+\eta \| p_F \|_{0,\Gamma}\\ \geq &
- 2 \langle 2 \mu \mathbf{D}\left(\fluidvelocity\right) \mathbf{n} \cdot \mathbf{n}+\frac 1 2 (\operatorname{tr} \boldsymbol{\sigma}_f),\frac{\mu_{F} d_{M}}{K_{\mathcal{M}}} \fluidvelocity \cdot \mathbf{n} \rangle_{0,\Gamma}
- 2 \langle p_F,\filtervelocity \cdot \mathbf{n}_{0,\Gamma} \rangle\\&+\frac{1}{\eta|\Gamma|}\left((1-\rho)\left(\int_{\Gamma} p d s\right)^{2}-\frac{1-\rho}{\rho}\left(\int_{\Gamma} \mathbf{n} \cdot(\boldsymbol{\sigma_f} \cdot \mathbf{n}) d s\right)^{2}\right)\end{aligned}
\end{equation}
for any sufficiently small $\eta>0$ and $\rho \in (0,1)$. This can be inserted in the lower bounds for the two functional of the domains $\Omega_p$ and $\Omega_f$ and completes the proof.
\end{proof}
A further major advantage of this result is the inherent error estimator for any conforming discretization, in particular for discretization in the space
\begin{equation}
\mathbb H _h :=
(RT^{k-1}(\Omega_f))^2 \times
P^k(\Omega_f)^2  \times
 RT^{k-1}(\Omega_p) \times
P^k(\Omega_p)
\end{equation}
for any $k \geq 1$, as considered in the next section.

\section{Numerical Results}
This section is concerned with the validation of the Least-Squares approach by numerical results. The first numerical example describes a capillary of length 1 mm before the bifurcation and surrounded by the tissue of the capillary bed,  in analogy to \cite{Koch2014CouplingAV}. Since the different terms of the Least-Squares Functional have different scaling, the domain is rescaled such that $L=2$ and an auxiliary solution is computed. A simple mapping transforms the solution to the original domain back. The boundary conditions are chosen from  \cite{Koch2014CouplingAV}: the blood is flowing in the capillary through $\Gamma_{D,f,1}$ where the effective pressure is set to $400$ Pa to $\Gamma_{D,f,2x}$ and $\Gamma_{D,f,3}$  where it is set to $-1600$  Pa. On $\Gamma_{D,F}$, the effective pressure $p_F$ is set to $-933$ Pa. The thickness of the vessel wall $\varepsilon$ is set to $6.0 \cdot 10^{-7}$m. The second numerical example is constructed in a similar way, but takes the structure of the liver lobule into account. In both cases, the results are shown in figure \ref{fig:3} and confirm the convergence of the method. The Least-Squares functional as an error estimator leads to a mesh refinement around the interface as expected. 
%The free-flow domain $\Omega_f$ , the surrounding tissue $\Omega_p$ and the applied boundary conditions are shown in Figure 7.3. 
\begin{figure}
\begin{minipage}{52mm}
\includegraphics[width=42mm]{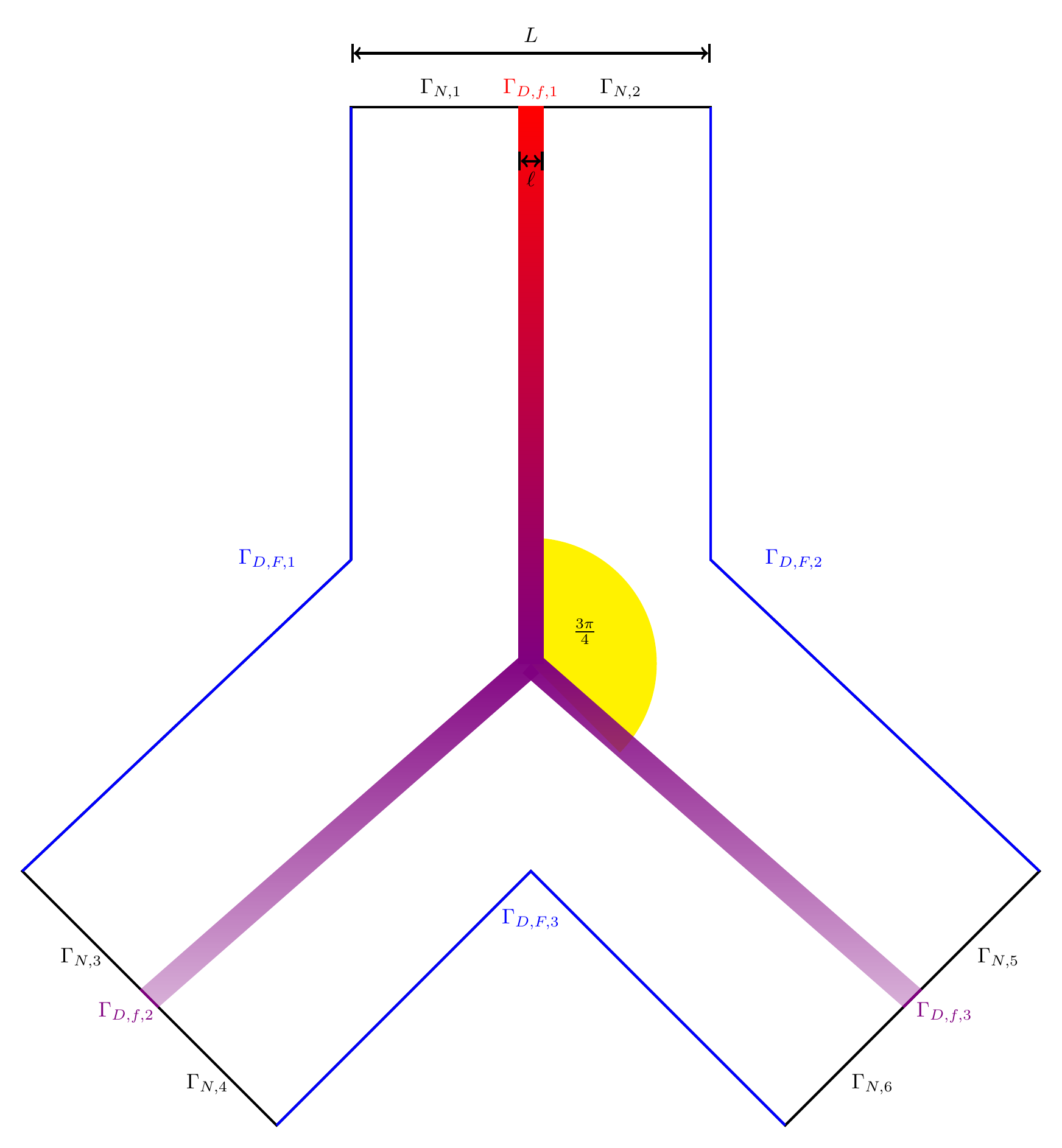}
\caption{Geometry for the bifurcation case.}
\label{fig:8}
\end{minipage}
\begin{minipage}{61mm}
\includegraphics[width=\linewidth]{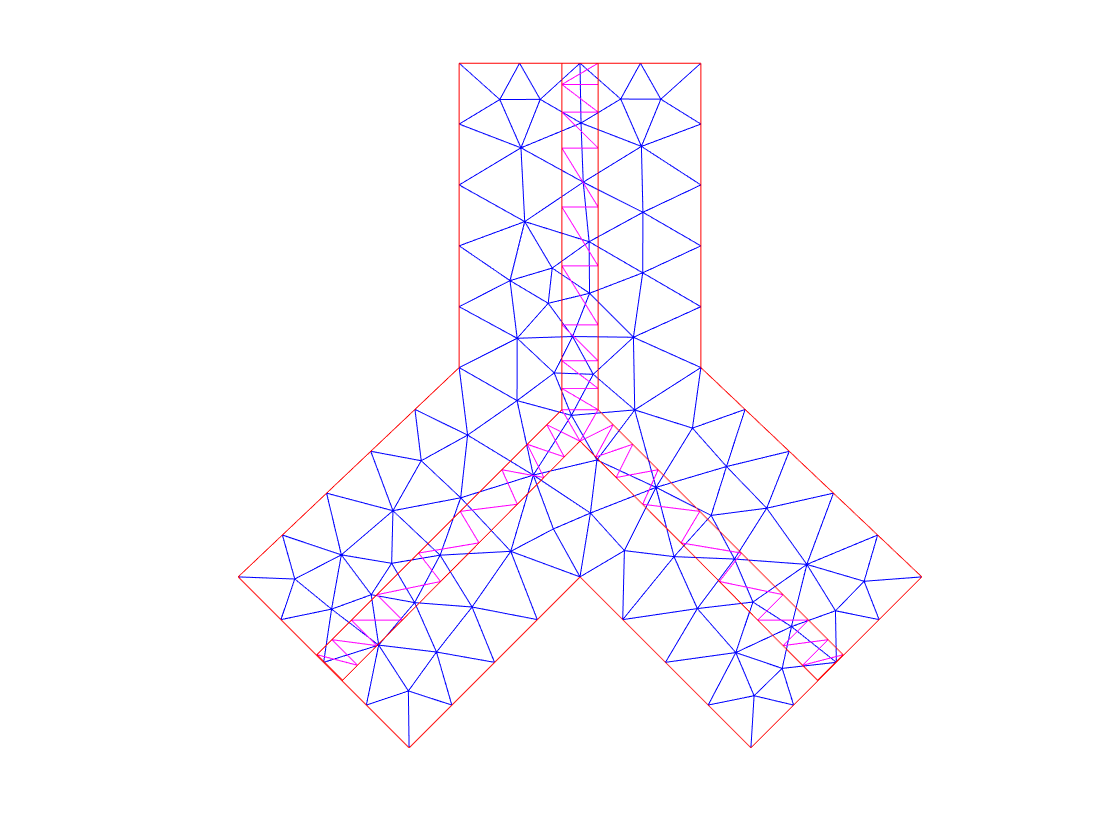}
\caption{Initial meshes for $\Omega_f$ and $\Omega_p$}
\label{fig:3}
\end{minipage}
\begin{minipage}{61mm}
\includegraphics[width=\linewidth]{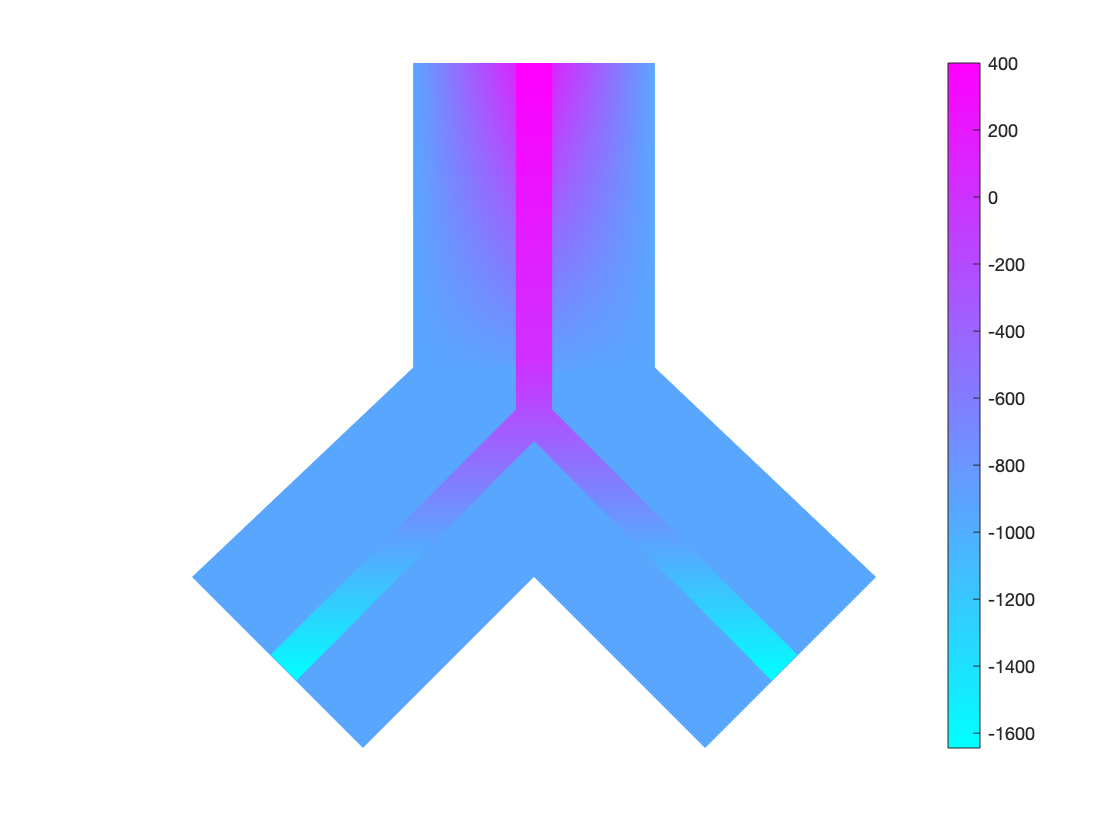}
\caption{Effective pressure %in the bifurcation case.
}
\label{fig:4}
\end{minipage}
%\hfil
%\begin{minipage}{52mm}
%\includegraphics[width=\linewidth]{mesh1}
%\caption{This is the second picture.}
%\label{fig:4}
%\end{minipage}
\end{figure}
\begin{figure}
\hfil
\begin{minipage}{72mm}
\includegraphics[width=\linewidth]{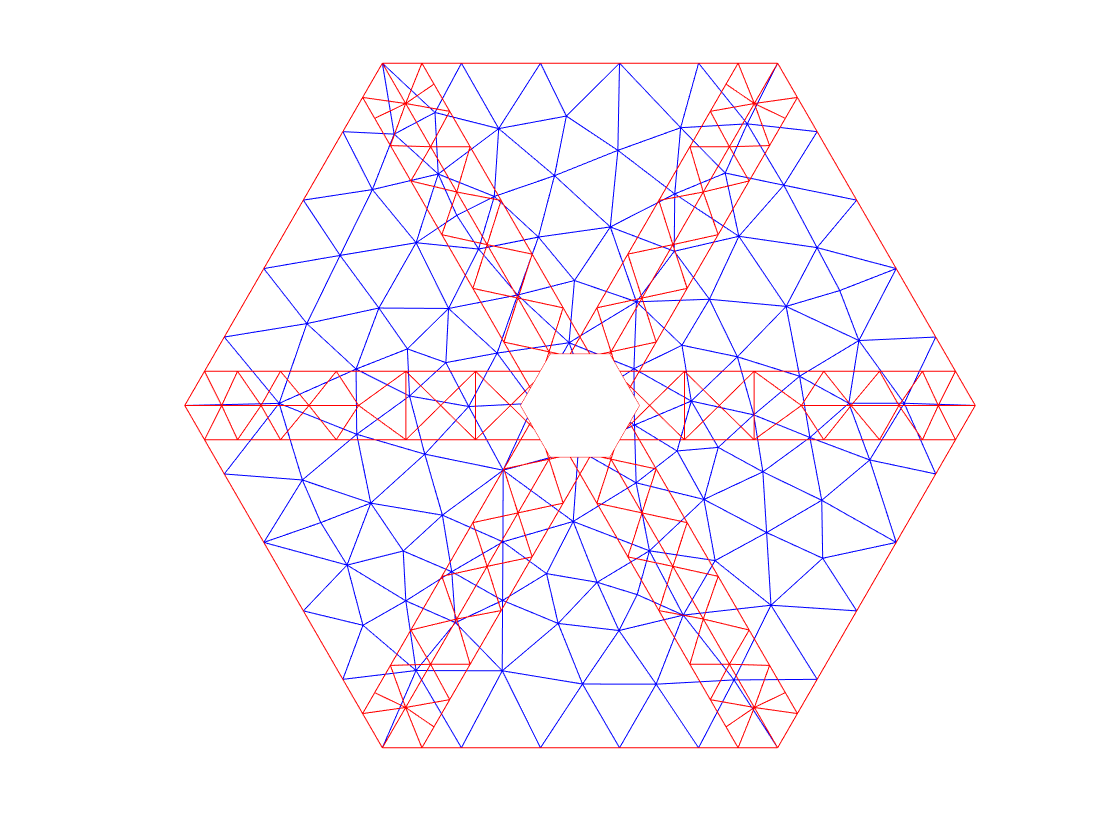}
\caption{Initial meshes for $\Omega_f$ and $\Omega_p$ representing the hexagonal-shaped liver lobule and the sinusoids therein.}
\label{fig:3}
\end{minipage}
\hfil
\begin{minipage}{72mm}
\includegraphics[width=\linewidth]{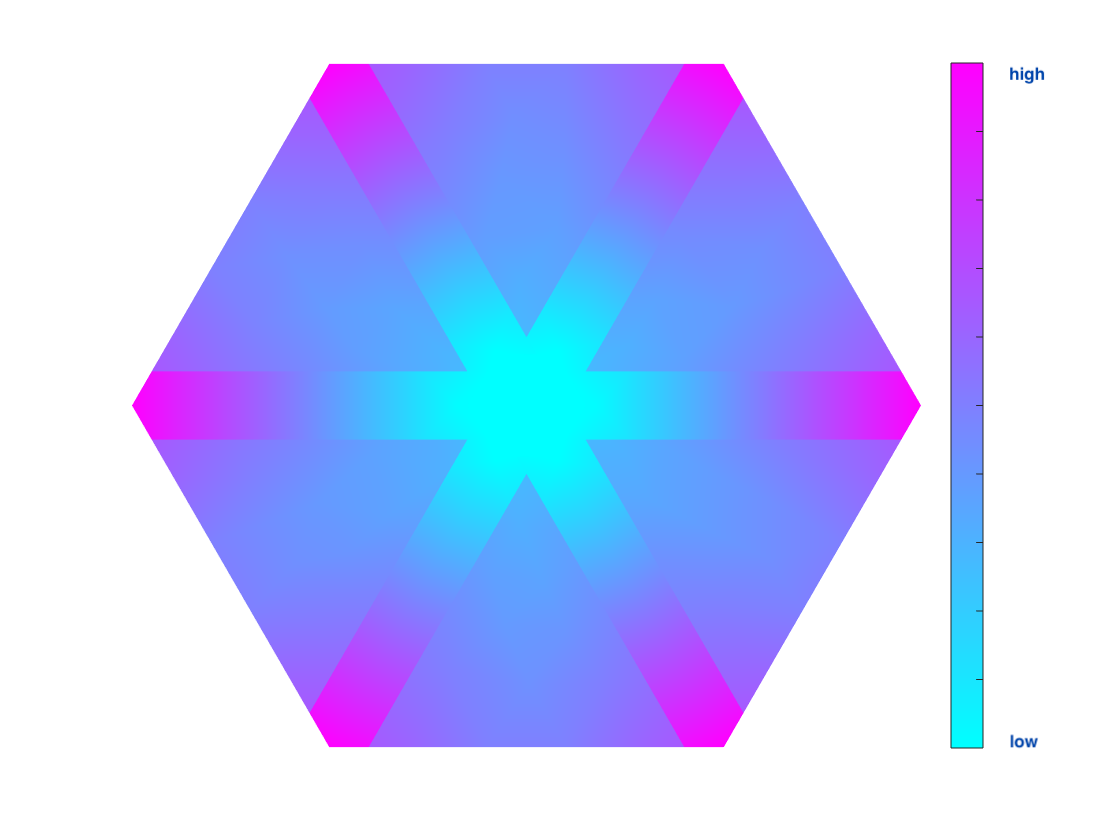}
\caption{Effective pressure in the liver lobule with periportal inflow and outflow at the central vein.}
\label{fig:4}
\end{minipage}
%\hfil
%\begin{minipage}{52mm}
%\includegraphics[width=\linewidth]{mesh1}
%\caption{This is the second picture.}
%\label{fig:4}
%\end{minipage}
\end{figure}
Acknowledgement:
  Funded by Deutsche Forschungsgemeinschaft (DFG, German Research Foundation) under Germany's Excellence Strategy – EXC 2075 – 390740016

\def\bstname{pamm}

% Use this code if you wish to generate your bibliography with BibTeX;
% please replace first the string "demo" below with the name(s) of
% the BibTeX data base(s) you want to use.
% The resulting bibliography-output (the contents of the .bbl file)
% must be pasted into this file before submission.
% 
% \bibliographystyle{pamm}
% \bibliography{demo}
% 
% Replace the following example bibliography with your references
% before submission:

%\listoffixmes
\end{document}